\documentclass[11pt,a4paper]{article}
\usepackage[T1]{fontenc}
\usepackage[utf8]{inputenc}
\usepackage{lmodern}
\usepackage[margin=28mm]{geometry}
\usepackage{amsmath,amssymb,amsthm,mathtools}
\usepackage{microtype}
\usepackage[hidelinks]{hyperref}
\usepackage{float}
\usepackage{bookmark}
\hypersetup{
  pdftitle={Mixing time under monotone censoring: Proof of the Ding-Mossel conjecture},
  pdfsubject={Mixing times, increasing subsets of the discrete cube, and killed Green functions},
  pdfauthor={Yiming Chen, with an Appendix by Yuval Peres},
  pdfkeywords={monotone censoring, hypercube, mixing time, Green function, hypercontractivity}
}
\allowdisplaybreaks[2]
\numberwithin{equation}{section}
\newtheorem{theorem}{Theorem}[section]
\newtheorem{proposition}[theorem]{Proposition}
\newtheorem{lemma}[theorem]{Lemma}

\theoremstyle{definition}

\theoremstyle{remark}

\newtheorem{question}[theorem]{Question}
\newcommand{\E}{\mathbb E}
\newcommand{\Prob}{\mathbb P}
\newcommand{\R}{\mathbb R}

\newcommand{\one}{\mathbf 1}
\newcommand{\cube}{\Omega_n}
\newcommand{\cE}{\mathcal E}
\newcommand{\Var}{\operatorname{Var}}
\newcommand{\supp}{\operatorname{supp}}

\newcommand{\tmix}{t_{\mathrm{mix}}}
\newcommand{\thit}{t_{\mathrm H}}

\newcommand{\ip}[3]{\langle #1,#2\rangle_{#3}}
\newcommand{\norm}[2]{\lVert #1\rVert_{#2}}
\newcommand{\dd}{\,\mathrm d}

\usepackage{tikz}

\definecolor{ink}{HTML}{273340}
\definecolor{muted}{HTML}{7C8792}
\definecolor{gridgray}{HTML}{DCE2E7}
\definecolor{openblue}{HTML}{427EB5}
\definecolor{pathred}{HTML}{C8494D}
\title{Mixing time under monotone censoring} 

\author{
\textsc{Yiming Chen}\thanks{
School of Mathematical Sciences, Peking University,
Beijing 100871, China.
Email: \texttt{ymchenmath@math.pku.edu.cn}
}
\\[0.45em]
{\normalsize\normalfont\textsc{with an Appendix by Yuval Peres}\thanks{
Beijing Institute of Mathematical Sciences and Applications (BIMSA),
Beijing, China.
Email: \texttt{yperes@gmail.com}
}}
}

\date{}

\begin{document}
\maketitle

\begin{abstract}
We prove the conjecture of Ding and Mossel~\cite{DM}: the lazy
random walk on the discrete cube, censored to any increasing set
of fixed positive density, mixes in time $O(n\log n)$. More precisely,
for every nonempty increasing set $A\subseteq\{0,1\}^n$, with
$a=\mu(A)$ and $\mu$ uniform on the cube, we prove
$$
 t_{\mathrm{mix}}(P)\le C_a n\log(en),
$$
where $C_a$ depends only on $a$. The proof uses hypercontractivity
on the ambient cube to strengthen the Poincar\'e inequality on
coordinate sections. An optimal stopping time occupation inequality
for increasing sets converts the resulting local bound into a
uniform bound on hitting times of large sets. An appendix by
Y.~Peres improves the density dependence to
$C_a=13a^{-1}\log(e/a)$, within a logarithmic factor of the
conjectured $a^{-1}$ dependence.
\end{abstract}

\medskip
\noindent\textbf{Keywords.}
Mixing times; monotone censoring; 
hypercontractivity; hitting times.

\medskip
\noindent\textbf{2020 MSC.}
Primary 60J10; Secondary 05C81, 60J45, 60J46.

\section{Main result and setup}\label{sec:setup}

The lazy random walk on $\{0,1\}^n$ mixes in time of order $n\log n$.
If moves leaving a subset $A$ are rejected, the resulting walk
converges to the uniform measure on $A$ whenever the induced graph
is connected, but its mixing time may be exponentially larger
\cite[Section~1]{FFP}. We show that the $O(n\log n)$ bound continues
to hold whenever $A$ is increasing and its density is bounded away
from zero.

Let $\Omega_n=\{0,1\}^n$ carry the uniform probability measure $\mu$,
and equip it with the coordinatewise order. A set $A\subseteq\Omega_n$
is \emph{increasing} if $x\in A$ and $x\leq y$ imply $y\in A$.
Throughout, $A$ is nonempty and increasing, and
$\pi=\mu(\,\cdot\mid A)$ denotes the uniform probability measure on
$A$. Write $a=\mu(A)=2^{-n}|A|$ and $\pi(x)=|A|^{-1}$ for $x\in A$.

The censored walk updates a uniformly chosen coordinate to an
independent fair bit and accepts the update if the resulting
configuration belongs to $A$, otherwise it stays at its current
state. Writing $x^{\oplus i}$ for $x$ with its $i$th coordinate
flipped, its transition kernel is
\begin{equation}\label{eq:kernel}
 P(x,y)=
 \begin{cases}
  \dfrac{1}{2n},
    & y=x^{\oplus i}\in A\text{ for some }i\in[n],\\[4pt]
  1-\dfrac{\deg_A(x)}{2n}, & y=x,\\[4pt]
  0, & \text{otherwise},
 \end{cases}
\end{equation}
where $\deg_A(x)=|\{i\in[n]:x^{\oplus i}\in A\}|$.
The kernel is symmetric and satisfies $P(x,x)\geq 1/2$, so the walk
is lazy and reversible with respect to $\pi$. It is irreducible
because every point of $A$ can be joined to the all-one configuration
by a path of upward moves.

For a probability measure $\nu$ on $A$, write
$\|\nu-\pi\|_{\mathrm{TV}}
 =\frac12\sum_{y\in A}|\nu(y)-\pi(y)|$.
The mixing time is defined by
\begin{equation}\label{eq:tmix}
 t_{\mathrm{mix}}(P)
 =\min\left\{t\in\mathbb Z_{\geq0}:
       \max_{x\in A}\|P^t(x,\cdot)-\pi\|_{\mathrm{TV}}
       \leq\frac14\right\}.
\end{equation}

Ding and Mossel~\cite{DM} initiated the study of this walk. By
combining a directed isoperimetric inequality from monotonicity
testing with the FKG inequality, they obtained a conductance lower
bound of order $ a/n$. Cheeger's inequality then gives the mixing
bound $O(n^3 a^{-2})$, and hence $t_{\mathrm{mix}}(P)=O_c(n^3)$ whenever
$ a\geq c>0$. Moreover, they posed the following question
\cite[Question~1.5]{DM}.

\begin{question}\label{question:DM}
For every $c\in(0,1]$, is there a constant $C(c)<\infty$ such that
$t_{\mathrm{mix}}(P)\leq C(c)n\log n$ for every $n\geq2$ and every
increasing $A\subseteq\{0,1\}^n$ with $ a\geq c$?
\end{question}

Cohen~\cite[Corollary~4.8, p.~41]{Cohen} obtained a relaxation time
bound of order $n^2/a$, and consequently a mixing bound
$O(n^2\log|A|/a)$. Her thesis also conjectures the stronger density
dependence $O(n\log n/a)$~\cite[p.~36]{Cohen}.
The conjecture and partial progress were subsequently discussed at
the Banff focused research group on concentration, relaxation and
mixing for restricted lattices in 2019~\cite{Banff}.

Fei and Ferreira Pinto Jr.~\cite{FFP} improved the mixing bound to
$O(n^2/a)$ by proving a spectral gap lower bound of order
$ a/n$. Their argument develops a directed $L^2$ Poincar\'e
inequality and an approximate FKG inequality. Chang, Sun, and
Yu~\cite{CSY} subsequently gave an elementary proof by induction on
coordinate restrictions, with a factor of two loss in the Poincar\'e
constant and the same order of mixing bound. 

A logarithmic Sobolev inequality
$\operatorname{Ent}_\pi(f^2)\leq C(c)n\mathcal E_A(f)$, where
$\mathcal E_A$ is the Dirichlet form of $P$, would give the desired
$O_c(n\log n)$ bound. Oliveira Santos, Tripathi,
and Youssef~\cite{OSTY} proved that such an inequality holds with high
probability when $A$ is chosen uniformly from the collection of
increasing subsets of the cube, thereby establishing the conjectured
mixing bound for a typical increasing set. On the other hand, they
also constructed increasing sets of density bounded away from zero
whose logarithmic Sobolev constant is of order $n^2$. Hence, a
logarithmic Sobolev inequality with constant $O_c(n)$ cannot hold
uniformly over all increasing sets of density at least $c>0$.

Our main theorem answers Question~\ref{question:DM} for every
increasing set. Let $(X_t)_{t\geq0}$ be the chain with kernel $P$,
and let $\mathbb E_x$ denote expectation when $X_0=x$. For
$B\subseteq A$, let $\tau_B=\inf\{t\in\mathbb Z_{\geq0}:X_t\in B\}$ denote the hitting time of $B$, and define
$$
t_H^P(1/4)
=\max_{\substack{x\in A,\ B\subseteq A\\ \pi(B)\geq1/4}}
\mathbb E_x\tau_B.
$$

\begin{theorem}\label{thm:main}
For every $n\geq1$ and every nonempty increasing set
$A\subseteq\{0,1\}^n$, the censored walk satisfies
%the kernel~\eqref{eq:kernel} 
\begin{equation}\label{eq:main-hitting}
 t_H^P(1/4)\leq 6148\, a^{-3}n\log(en).
\end{equation}
Moreover,
\begin{equation}\label{eq:main-mixing}
 t_{\mathrm{mix}}(P)\le C_a n\log(en),
\end{equation}
where $C_a=\frac{13}{a}\log\frac ea.$
\end{theorem}

In particular, the density assumption in Question~\ref{question:DM}
is sufficient to recover the mixing bound of the full cube.
We give two complementary arguments. The main argument constructs a
single increasing core on which local functional inequalities hold
at every scale. This yields \eqref{eq:main-hitting} and, through the
hitting-time characterization of mixing, resolves the fixed density
Question~\ref{question:DM} with a bound of order $a^{-3}n\log(en)$. The appendix
replaces the fixed core by noise averaged weights that vary with the
scale, giving \eqref{eq:main-mixing}. Both approaches combine section Poincar\'e inequalities, occupation potentials, and truncations of
the killed Green function. The core construction makes the first
route transparent, while the varying weights give
the better dependence on $a$.

%\begin{corollary}\label{cor:constant-density}
%For every $c\in(0,1]$, there is a constant $C(c)<\infty$ such that,
%for every $n\geq2$ and every increasing $A\subseteq\{0,1\}^n$ with
%$ a\geq c$,
%\begin{equation}\label{eq:constant-density}
% t_{\mathrm{mix}}(P)\leq C(c)n\log n.
%\end{equation}
%\end{corollary}

A concrete motivating example, emphasized in the published paper of
Ding and Mossel~\cite[Section~1.1]{DM}, comes from critical site
percolation on the hexagonal tiling. Take the standard symmetric
rhombus with $L$ hexagons along each side, so that $n=L^2$ is the
number of hexagons, and assign independent fair bits to its cells.
Let $A$ be the event of a left-to-right crossing of cells with bit
one. Exactly one of this event and a top-to-bottom crossing of cells
with bit zero occurs. Reflection interchanging the two pairs of
opposite sides, followed by color exchange, preserves the product
law; hence $a=1/2$.

The dynamics chooses a hexagon uniformly, refreshes its bit, and
accepts the refresh precisely when a left-to-right one-crossing
remains. Its stationary law is $\mu(\,\cdot\mid A)$.
An update is local, but its admissibility depends on the entire
configuration: a hexagon pivotal for the crossing cannot be closed.
Theorem~\ref{thm:main} gives an $O(n\log n)$ mixing bound.
Proposition~\ref{prop:universal-lower} below shows that this order is
sharp.

\begin{figure}[htbp]
\centering
\begin{tikzpicture}[scale=1.12, line join=round,
  every node/.style={font=\small,text=ink}]
\def\rad{0.38}
\foreach \q in {0,...,5}{
 \foreach \r in {0,...,5}{
  \pgfmathsetmacro{\xx}{sqrt(3)*\rad*(\q+\r/2)}
  \pgfmathsetmacro{\yy}{1.5*\rad*\r}
  \def\cellfill{openblue}
  \ifnum\r=2 \def\cellfill{pathred}\fi
  \pgfmathtruncatemacro{\isred}{
   (\r==2) || (\r==0 && (\q==0 || \q==2 || \q==5)) ||
   (\r==1 && (\q==1 || \q==4)) ||
   (\r==3 && (\q==0 || \q==3 || \q==5)) ||
   (\r==4 && (\q==1 || \q==4)) ||
   (\r==5 && (\q==2 || \q==3))}
  \ifnum\isred=1 \def\cellfill{pathred}\fi
  \path[draw=white,fill=\cellfill,line width=0.65pt]
   (\xx,\yy) ++(30:\rad)
   \foreach \ang in {90,150,210,270,330}{-- ++(\ang+60:\rad)}
   -- cycle;
 }
}
\node[anchor=east] at (0.25,1.48) {left};
\node[anchor=west] at (4.56,1.48) {right};
\node[anchor=south] at (3.26,3.35) {top};
\node[anchor=north] at (1.65,-0.48) {bottom};
\end{tikzpicture}
\caption{A configuration in the crossing event $A$ on a hexagonal
rhombus. Every hexagon is red (bit $1$) or blue (bit $0$).
The red hexagons contain a left-to-right crossing. The dynamics
accepts a refresh exactly when such a crossing remains.}
\label{fig:ding-mossel-crossing}
\end{figure}

\subsection{Proof outline}

We first construct an increasing set \(C\subseteq A\) with \(\mu(C)\ge a/2\) such that, for every \(x\in C\) and every noise scale, the random coordinate section through \(x\) has density at least \(a/8\) with probability at least \(3/4\). Applying the resulting local functional inequality to dyadic truncations of the killed Green function \(G\) gives
\[
\int_C G\,d\mu
\le
\frac{\mu(C)}{2}\int_A G\,d\mu
+ O_a\!\bigl(n\log(en)\bigr).
\]

On the other hand, an occupation inequality gives
\[
\mu(C)\int_A G\,d\mu
\le
\int_C G\,d\mu
+ O\!\bigl(an\log(en)\bigr).
\]

Combining the two bounds with $\int_A G\,d\mu = a\,\mathbb E_x\tau_B$ yields \(\mathbb E_x\tau_B=O_a(n\log n)\). The Peres--Sousi~\cite{PS} hitting time characterization then gives \(t_{\mathrm{mix}}(P)=O(a^{-3}n\log(en))\), answering Question~\ref{question:DM}. 
%The appendix replaces the fixed core
%by noise averaged weights at varying scales and proves
%\eqref{eq:main-mixing} with the stated $C_a$.

\subsection{Notation}

For functions on $A$, write $\pi(f)=\sum_{x\in A}\pi(x)f(x)$,
$\langle f,g\rangle_\pi=\pi(fg)$, and
$\operatorname{Var}_\pi(f)=\pi(f^2)-\pi(f)^2$.
The Dirichlet energy is
\begin{align}
 \mathcal E_A(f,g)
 &=\langle f,(I-P)g\rangle_\pi\notag\\
 &=\frac1{4n}\sum_{x\in A}\pi(x)
   \sum_{\substack{i\in[n]\\x^{\oplus i}\in A}}
   \bigl(f(x)-f(x^{\oplus i})\bigr)
   \bigl(g(x)-g(x^{\oplus i})\bigr).
 \label{eq:dirichlet}
\end{align}
We abbreviate $\mathcal E_A(f,f)$ to $\mathcal E_A(f)$.
Integrals over subsets of $A$ with respect to $\mu$ use the
unnormalized restriction of the ambient measure, in particular,
\begin{equation}\label{eq:mu-pi}
 \int_A f\,d\mu= a\pi(f).
\end{equation}
For a measure $\nu$ and $1\leq p<\infty$, we use
$\|f\|_{p,\nu}=(\int|f|^p\,d\nu)^{1/p}$.
When $\nu=\mu$, the domain is the whole cube unless otherwise
indicated. For $f:A\to\mathbb R$, its zero extension to $\Omega_n$
is denoted by $\widetilde f$. All logarithms are natural.

Section~\ref{sec:inputs} introduces the preliminary inequalities.
Sections~\ref{sec:core}--\ref{sec:main-proof} carry out the argument for \eqref{eq:main-hitting} and the bound of order $a^{-3}n\log(en)$. Section~\ref{sec:concluding} shows that the
$n\log n$ is sharp uniformly at fixed positive density, and the
appendix proves the \eqref{eq:main-mixing} in
Theorem~\ref{thm:main}.

\section{Preliminaries}
\label{sec:inputs}

\subsection{The Poincar\'e inequality for increasing sets}

We use the following theorem in the normalization of
$\cE_A$.

\begin{theorem}[Fei--Ferreira Pinto Jr.~{\cite[Definition~1.5 and Theorem~1.6]{FFP}}]
\label{thm:FFP}
Let $A\subseteq\Omega_n$ be nonempty and increasing. Then every
real-valued function $f$ on $A$ satisfies
\begin{equation}
 \Var_{\pi}(f)
 \leq
 \frac{n}{1-\sqrt{1- a}}\,\cE_A(f).
 \label{eq:external-poincare}
\end{equation}
\end{theorem}

Since
\begin{equation}
 1-\sqrt{1- a}
 =
 \frac{a}{1+\sqrt{1- a}}
 \geq \frac{a}{2},
 \label{eq:density-denominator}
\end{equation}
\eqref{eq:external-poincare} gives
\begin{equation}
 \Var_\pi(f)
 \leq
 \frac{2n}{a}\cE_A(f).
 \label{eq:global-poincare}
\end{equation}

For later use we describe the coordinate sections explicitly.
For $R\subseteq[n]$ and $z\in\{0,1\}^{R^c}$, define
$A_{R,z}=\{u\in\{0,1\}^{R}:(u,z)\in A\}$, and let $\mu_{R,z}$
be the uniform probability measure on the full coordinate section
with frozen coordinates $z$. Thus, for $h:\Omega_n\to\mathbb R$,
\[
 \mu_{R,z}(h)
 =2^{-|R|}\sum_{u\in\{0,1\}^{R}}h(u,z).
\]
In particular, $\mu_{R,z}(A)=2^{-|R|}|A_{R,z}|$.
When $A_{R,z}$ is nonempty, let
$\pi_{R,z}=\mu_{R,z}(\,\cdot\mid A)$.

For a function $f$ on $A$, define
\begin{equation}
 \cE_{R,z}(f)
 =
 \frac1{4n}
 \sum_{\substack{x\in A\\x_{R^c}=z}}
 \pi_{R,z}(x)
 \sum_{\substack{i\in R\\x^{\oplus i}\in A}}
 (f(x)-f(x^{\oplus i}))^2.
 \label{eq:section-energy}
\end{equation}
The denominator is the original dimension $n$, even when $|R|<n$.
For an empty section, set $\cE_{R,z}(f)=0$. Then we can apply the Poincar\'e inequality in Theorem~\ref{thm:FFP}  to each section.

\begin{lemma}
\label{lem:section-poincare}
For every nonempty section and every real-valued function $f$ on $A$,
\begin{equation}
 \Var_{\pi_{R,z}}(f)
 \leq
 \frac{2n}{\mu_{R,z}(A)}\cE_{R,z}(f).
 \label{eq:section-poincare}
\end{equation}
\end{lemma}

\begin{proof}

The set $A_{R,z}$ is increasing in its free coordinates. Indeed,
increasing the coordinates in $R$ increases the corresponding point
of the full cube while leaving the coordinates in $R^c$ fixed, so
membership in $A$ is preserved. If $R=\varnothing$, every nonempty section consists of a single
point, and both sides of \eqref{eq:section-poincare} are zero.

Suppose now that $R\neq\varnothing$. Apply
Theorem~\ref{thm:FFP} to the increasing set
$A_{R,z}\subseteq\{0,1\}^{R}$, identifying it with
$A\cap\{x\in\Omega_n:x_{R^c}=z\}$. Under this identification, we have
\[
 \frac14
 \sum_{\substack{x\in A\\x_{R^c}=z}}
 \pi_{R,z}(x)
 \sum_{\substack{i\in R\\x^{\oplus i}\in A}}
 (f(x)-f(x^{\oplus i}))^2
 =
 n\cE_{R,z}(f).
\]
Therefore
\[
 \Var_{\pi_{R,z}}(f)
 \le
 \frac{n}
 {1-\sqrt{1-\mu(A\mid x_{R^c}=z)}}
 \cE_{R,z}(f)
 \le
 \frac{2n}{\mu(A\mid x_{R^c}=z)}
 \cE_{R,z}(f),
\]
where the last inequality follows from
\[
 1-\sqrt{1-\mu(A\mid x_{R^c}=z)}
 \ge
 \frac12\mu(A\mid x_{R^c}=z).
\]

\end{proof}

\subsection{Hitting large sets and mixing}

For a finite irreducible Markov kernel $Q$ with stationary
distribution $\nu$, let $Q_{\mathrm L}=\frac{I+Q}{2}$ be its lazy version, which eliminates periodicity. Define
\[
 \thit^Q(\alpha)
 =\max_{\substack{x,\ B\\ \nu(B)\ge\alpha}}
       \E_x^Q\tau_B.
\]

Peres--Sousi~\cite{PS} show that the mixing time can be characterized by the time to hit sufficiently large sets.
\begin{theorem}[\cite{PS}]
\label{thm:PS}

For every fixed $\alpha\in(0,1/2)$ there exists a finite constant
$C_\alpha$, depending only on $\alpha$, such that every finite
irreducible reversible Markov kernel $Q$ satisfies
\begin{equation}
 \tmix(Q_{\mathrm L})
 \le
 C_\alpha\thit^Q(\alpha).
 \label{eq:PS}
\end{equation}
Here the mixing threshold is $1/4$. 

\end{theorem}

\subsection{Hypercontractivity on the discrete cube}

For $0\leq\rho\leq1$ and $x\in\Omega_n$, let $Y$ be obtained
from $x$ by independently retaining each coordinate with
probability $\rho$ and otherwise replacing it by a fresh fair bit.
Define $T_\rho f(x)=\mathbb E[f(Y)]$. The operator is self-adjoint
on $L^2(\mu)$, and two successive noise steps satisfy
$T_\rho T_\sigma=T_{\rho\sigma}$. The Bonami--Beckner inequality
\cite{Bonami,Beckner} states that
\begin{equation}
 \norm{T_\rho f}{2,\mu}
 \le\norm{f}{1+\rho^2,\mu}.
 \label{eq:hypercontractivity}
\end{equation}
We will apply this to zero extensions of functions on $A$.

\section{An increasing core obtained by random revelations}
\label{sec:core}

\subsection{Construction and density of the core}

Revealing coordinates changes the conditional density of $A$.
We retain points for which a random revelation order is unlikely
to make this density small at any stage. This single condition
will bound random sections of every size.

Set $\delta=\frac{a}{8}.$ Given a fixed permutation $\sigma = (\sigma(1), \ldots, \sigma(n)),$ define
\begin{equation}
 a_k^\sigma(x)
 =
 \mu\bigl(
     A
     \mid
     U_{\sigma(1)}=x_{\sigma(1)},\ldots,
     U_{\sigma(k)}=x_{\sigma(k)}
 \bigr),
 \qquad x\in\Omega_n,
 \label{eq:reveal-density}
\end{equation}
where $U\sim\mu$. In particular, $a_0^\sigma(x)= a,$ $a_n^\sigma(x)=\one_A(x).$ Taking probability over a uniformly random permutation $\sigma$, define

\[
 b(x)
 =
 \Prob_\sigma\left(
     \min_{0\le k\le n}a_k^\sigma(x)<\delta
 \right),
 \quad \text{and} \quad
 C=\{x\in\Omega_n:b(x)\le1/4\}.
\]

\begin{lemma}
\label{lem:core-mass}
The core $C$ is increasing, satisfies $C\subseteq A$, and
\begin{equation}
 \mu(C)\ge\frac{a}{2}>0.
 \label{eq:core-mass}
\end{equation}
\end{lemma}

\begin{proof}
Fix $\sigma$ and $k$. Suppose that $x\le y$. Couple the unrevealed
coordinates in the two conditional distributions in
\eqref{eq:reveal-density} to be identical. The resulting
configuration with revealed coordinates taken from $x$ is then
coordinatewise no larger than the corresponding configuration with
revealed coordinates taken from $y$. Since $A$ is increasing, $a_k^\sigma(x)\le a_k^\sigma(y).$ Hence
\[
 \left\{
     \min_{0\le k\le n}a_k^\sigma(y)<\delta
 \right\}
 \subseteq
 \left\{
     \min_{0\le k\le n}a_k^\sigma(x)<\delta
 \right\},
\]
and therefore $b(y)\le b(x)$. Thus $b$ is decreasing and $C$ is
increasing.

If $x\notin A$, then $a_n^\sigma(x)=\one_A(x)=0<\delta$ for every $\sigma$. Hence $b(x)=1$, so $x\notin C$. Therefore
$C\subseteq A$.

It remains to estimate the mass of $C$. Fix $\sigma$, and let
$\mathcal F_k$ be the $\sigma$-field generated by
$U_{\sigma(1)},\ldots,U_{\sigma(k)}$. By definition,
\[
 a_k^\sigma(U)
 =
 \E_\mu\bigl[\one_A(U)\mid\mathcal F_k\bigr].
\]
Define
\[
 \zeta
 =
 \inf\bigl\{
     k\in\{0,\ldots,n\}:a_k^\sigma(U)<\delta
 \bigr\},
\]
with the convention $\inf\varnothing=\infty$. Since
$\{\zeta=k\}\in\mathcal F_k$, we have 
\begin{align*}
 \mu(A\cap\{\zeta\le n\})
 &=
 \sum_{k=0}^n
 \E_\mu\bigl[
     \one_{\{\zeta=k\}}\one_A(U)
 \bigr]\\
 &=
 \sum_{k=0}^n
 \E_\mu\bigl[
     \one_{\{\zeta=k\}}a_k^\sigma(U)
 \bigr]\\
 &\le
 \delta
 \sum_{k=0}^n
 \mu(\{\zeta=k\})
 \le
 \delta.
\end{align*}
All expectations in the preceding display are taken with respect to
the original measure $\mu$. Averaging over the uniform permutation
$\sigma$ yields
\[
 \int_A b(x)\,d\mu(x)\le\delta.
\]

On $A\setminus C$ one has $b(x)>1/4$. Hence
\[
 \mu(A\setminus C)
 \le
 4\int_A b(x)\,d\mu(x)
 \le
 4\delta
 =
 \frac{a}{2}.
\]
Since $C\subseteq A$,
\[
 \mu(C)
 =
  a-\mu(A\setminus C)
 \ge
 \frac{a}{2}.
\]
\end{proof}

The core in Lemma~\ref{lem:core-mass} has good random section properties at every noise scale.

\begin{lemma}
\label{lem:good-sections}
Fix $\theta\in[0,1]$, and let $R\subseteq[n]$ include each coordinate
independently with probability $\theta$. Then, for every $x\in C$,
\begin{equation}
 \Prob_R\bigl(\mu\bigl(A\mid U_{R^c}=x_{R^c}\bigr) \ge\delta\bigr)\ge\frac34.
 \label{eq:good-sections}
\end{equation}

\end{lemma}

\begin{proof}
Independently choose a uniform permutation $\sigma$ and a random integer $J\sim\mathrm{Bin}(n,1-\theta)$.
Set $R^c=\{\sigma(1),\ldots,\sigma(J)\}.$ For each fixed subset $S\subseteq[n]$ of size $k$,
\[
 \Prob(R^c=S)
 =
 \frac{\Prob(J=k)}{\binom nk}
 =
 (1-\theta)^k\theta^{n-k},
\]
with the usual endpoint interpretation. Thus $R^c$ has the
law of the complement of the independently selected set $R$. Under this coupling,
\[
 \mu\bigl(A\mid U_{R^c}=x_{R^c}\bigr)=a_J^\sigma(x).
\]
Hence, for every $x\in C$,
\begin{align*}
 \Prob_R\bigl(\mu\bigl(A\mid U_{R^c}=x_{R^c}\bigr)<\delta\bigr)
 &=
 \Prob_{\sigma,J}\bigl(a_J^\sigma(x)<\delta\bigr)\\
 &\le
 \Prob_\sigma\left(
     \min_{0\le k\le n}a_k^\sigma(x)<\delta
 \right)\\
 &=b(x)
 \le\frac14.
\end{align*}
Therefore
\[
 \Prob_R\bigl(\mu\bigl(A\mid U_{R^c}=x_{R^c}\bigr)\ge\delta\bigr)\ge\frac34.
\]
\end{proof}

\subsection{Averaging section energies and conditional means}

For fixed $R$, averaging over $z\in\{0,1\}^{R^c}$ is always with
respect to the uniform measure, so each $z$ has weight
$2^{-|R^c|}$. 

\begin{lemma}
\label{lem:energy-averaging}
Let $R\subseteq[n]$ be chosen as in
Lemma~\ref{lem:good-sections}. Then, for every
$f:A\to\mathbb R$,
\begin{equation}
 \E_R\left[
   \sum_{z\in\{0,1\}^{R^c}}2^{-|R^c|}
             \mu_{R,z}(A)\cE_{R,z}(f)\right]
 = a\theta\cE_A(f),
 \label{eq:energy-average}
\end{equation}
where $\mu_{R,z}(A)=\mu(A \mid U_{R^c} = z)$.
\end{lemma}

\begin{proof}
For every nonempty section,
\[
 \mu_{R,z}(A)\,\pi_{R,z}
 =
 \left.\mu_{R,z}\right|_A.
\]

Thus, for fixed $R$,
\begin{align*}
 \sum_{z\in\{0,1\}^{R^c}} 2^{-|R^c|}\mu_{R,z}(A)\cE_{R,z}(f)=\frac1{4n}\int_A
     \sum_{\substack{i\in R\\x^{\oplus i}\in A}}
           (f(x)-f(x^{\oplus i}))^2\dd\mu(x).
\end{align*}
Empty sections contribute zero by convention.

Since each coordinate $i$ belongs to $R$ with probability $\theta$, $\E_R\one_{\{i\in R\}}=\theta.$ Averaging over $R$ therefore yields

\[
 \frac\theta{4n}\int_A
     \sum_{\substack{i\in[n]\\x^{\oplus i}\in A}}
           (f(x)-f(x^{\oplus i}))^2\dd\mu(x)
 = a\theta\cE_A(f),
\]
by \eqref{eq:mu-pi} and \eqref{eq:dirichlet}.
\end{proof}

We next express the averaged squared section means in terms
of the cube noise operator.

\begin{lemma}
\label{lem:mean-averaging}
For any real-valued $h$ on $\cube$ and $\theta\in[0,1]$,
\begin{align}
 \E_R\sum_z2^{-|R^c|}
          \bigl(\mu_{R,z}(h)\bigr)^2=\norm{T_{\sqrt{1-\theta}}h}{2,\mu}^2
 \le\norm{h}{2-\theta,\mu}^2.
 \label{eq:mean-average}
\end{align}
\end{lemma}

\begin{proof}
Choose $R$ as above, choose the frozen coordinates $z$ uniformly,
and then choose $U,V$ independently and uniformly within that
section. Conditional on $(R,z)$,
$\E[h(U)h(V)]=\mu_{R,z}(h)^2$. Unconditionally, $U$ is uniform
on the cube. Given $U$, the point $V$ retains each coordinate
outside $R$ and refreshes each coordinate in $R$ independently.
Thus $V$ is obtained from $U$ by noise with parameter $1-\theta$,
and
\[
 \E_R\sum_z2^{-|R^c|}\mu_{R,z}(h)^2
 =\E[h(U)h(V)]
 =\langle h,T_{1-\theta}h\rangle_\mu
 =\|T_{\sqrt{1-\theta}}h\|_{2,\mu}^2.
\]
The last equality uses self-adjointness and the semigroup property.
Hypercontractivity~\eqref{eq:hypercontractivity} gives
$\|T_{\sqrt{1-\theta}}h\|_{2,\mu}\le\|h\|_{2-\theta,\mu}$.
\end{proof}

The \(L^2(\mu)\)-mass of \(f\) on the large core \(C\) can be bounded by two terms: one is the global Dirichlet energy of \(f\) on \(A\), and the other is a lower order \(L^{2-\theta}\)-norm of its zero extension.

\begin{proposition}
\label{prop:core-functional}
For every $f:A\to\R$ and every $0<\theta\le1$,
\begin{equation}
 \frac34\int_C f^2\,d\mu
 \le
 \frac{2 an\theta}{\delta}\cE_A(f)
 +
 \frac1\delta
 \|\widetilde f\|_{2-\theta,\mu}^{\,2}.
 \label{eq:core-functional}
\end{equation}
\end{proposition}

\begin{proof}
Write $\alpha=\mu_{R,z}(A)$. On a good section, where
$\alpha\ge\delta$, the variance decomposition and
Lemma~\ref{lem:section-poincare} give
\begin{equation}\label{eq:mu-bound}
\begin{split}
 \mu_{R,z}(\widetilde f^2)
 &=\alpha\Var_{\pi_{R,z}}(f)
      +\frac{\mu_{R,z}(\widetilde f)^2}{\alpha}\\
 &\le2n\cE_{R,z}(f)
      +\frac{\mu_{R,z}(\widetilde f)^2}{\alpha}\\
 &\le\frac{2n}{\delta}\alpha\cE_{R,z}(f)
      +\frac1\delta\mu_{R,z}(\widetilde f)^2.
\end{split}
\end{equation}

Average \eqref{eq:mu-bound} over the good sections. By
Lemma~\ref{lem:good-sections}, its left side is
\begin{align*}
 \E_R\sum_z2^{-|R^c|}\one_{\{\mu_{R,z}(A)\ge\delta\}}
      \mu_{R,z}(\widetilde f^2)
 &=\int_A f(x)^2
      \Prob_R\{\mu_{R,x_{R^c}}(A)\ge\delta\}\,d\mu(x)\\
 &\ge\frac34\int_C f^2\,d\mu.
\end{align*}
Both terms on the right are nonnegative, so their averages over
all sections are upper bounds for their averages over good sections.
Lemmas~\ref{lem:energy-averaging} and~\ref{lem:mean-averaging}
therefore give
\[
 \frac34\int_C f^2\,d\mu
 \le\frac{2n}{\delta}\,a\theta\cE_A(f)
    +\frac1\delta\|\widetilde f\|_{2-\theta,\mu}^{\,2},
\]
as required.
\end{proof}

The lower norm in \eqref{eq:core-functional} is small compared
with the $L^2$ norm when $f$ has small support. Later, this support
will be a high level set of the Green function.
\begin{lemma}
\label{lem:support-factor}
Suppose $f:A\to\R$ and
$\mu(\supp\widetilde f)\le v\le1$. For $0<\theta\le1$,
\begin{equation}
 \norm{\widetilde f}{2-\theta,\mu}^{\,2}
 \le v^{\theta/(2-\theta)}\int_A f^2\dd\mu.
 \label{eq:support-factor}
\end{equation}
\end{lemma}

\begin{proof}
By H\"older's inequality, with $p=2-\theta$,
\[
\|\widetilde f\|_{p,\mu}^{2}
\le
\mu(\operatorname{supp}\widetilde f)^{2/p-1}
\|\widetilde f\|_{2,\mu}^{2}
\le
v^{\theta/(2-\theta)}
\int_A f^2\,d\mu.
\]
\end{proof}

\section{An optimal occupation inequality at stopping times}
\label{sec:occupation}

In this section, we show that the expected occupation of any increasing
set \(C\subseteq A\) up to an integrable stopping time is at least
\(\mu(C)\) times the expected elapsed time, up to an additive error
depending only on \(n\).
%The following occupation inequality holds for every increasing set \(C\subseteq A\).
%The following proposition uses only that $C\subseteq A$ is increasing;
%it does not depend on its construction in Section~\ref{sec:core}.
%Thus, the expected occupation of $C$ up to any integrable
%stopping time is at least $\mu(C)$ times the expected elapsed time,
%apart from an error depending only on the dimension. 
Recall
\begin{equation}\label{eq:coupon-upper}
 D_n=n\sum_{j=1}^{\lceil n/2\rceil}\frac1{2j-1}
 \le n\left(1+\frac12\log n\right)\le n\log(en).
\end{equation}
The first inequality follows by comparing the terms after the first
with the integral of $1/(2u-1)$ from $1$ to $\lceil n/2\rceil$.

\begin{proposition}\label{prop:occupation}
Let $C\subseteq A$ be increasing. For every $x\in A$ and every
stopping time $T$ for the natural filtration of $(X_t)_{t\ge0}$
with $\mathbb E_xT<\infty$,
\begin{equation}\label{eq:occupation}
 \mu(C)\mathbb E_xT
 \le \mathbb E_x\sum_{t=0}^{T-1}\mathbf1_C(X_t)+D_n.
\end{equation}
The sum is zero when $T=0$, and $D_n$ is the smallest constant
valid uniformly over $A,C,x,T$.
\end{proposition}

\begin{proof}
Let $K$ be the uncensored single coordinate fair resampling kernel.
The coupling using common update coordinates and fair bits preserves
order and coalesces once every coordinate has been updated.
Comparing with a stationary chain and using the mean coupon
collector time gives
\[
 \sum_{t\ge0}\|K^t\mathbf1_C-\mu(C)\|_\infty
 \le n\sum_{j=1}^n\frac1j<\infty.
\]
Thus the following uniformly convergent series defines an increasing
potential:
\begin{equation}\label{eq:occupation-potential}
 V_C=\sum_{t\ge0}(K^t\mathbf1_C-\mu(C))
     =n\int_0^1(T_u\mathbf1_C-\mu(C))\,\frac{du}{u}.
\end{equation}
For the integral identity, the Walsh functions
$\chi_S(z)=\prod_{i\in S}(2z_i-1)$ satisfy
\[
 K\chi_S=(1-|S|/n)\chi_S,
 \quad \text{and} \quad T_u\chi_S=u^{|S|}\chi_S.
\]
Hence both expressions multiply each nonconstant coefficient by
$n/|S|$, while the constant coefficient vanishes.
Telescoping the series gives
\begin{equation}\label{eq:free-poisson}
 (I-K)V_C=\mathbf1_C-\mu(C).
\end{equation}

Since $V_C$ is increasing, its oscillation is
$V_C(1^n)-V_C(0^n)$.
Let
\[
 \mathcal M_n=\left\{z\in\Omega_n:
                  \sum_{i=1}^n z_i>n/2\right\}.
\]
For $0<u<1$, the product Bernoulli laws
$T_u(1^n,\cdot)$ and $T_u(0^n,\cdot)$ have likelihood ratio
\[
 \frac{T_u(1^n,z)}{T_u(0^n,z)}
 =\left(\frac{1+u}{1-u}\right)^{2\sum_i z_i-n}.
\]
This ratio exceeds one exactly on $\mathcal M_n$ and equals one
on ties. Therefore
\[
 \Delta_n(u):=
 \|T_u(1^n,\cdot)-T_u(0^n,\cdot)\|_{\mathrm{TV}}
 =T_u\mathbf1_{\mathcal M_n}(1^n)
  -T_u\mathbf1_{\mathcal M_n}(0^n),
\]
and \eqref{eq:occupation-potential} yields
\begin{equation}\label{eq:osc-VC}
 \begin{aligned}
 \operatorname{osc}V_C
 &=n\int_0^1
   \bigl(T_u\mathbf1_C(1^n)-T_u\mathbf1_C(0^n)\bigr)
   \frac{du}{u}\\
 &\le n\int_0^1\Delta_n(u)\,\frac{du}{u}.
 \end{aligned}
\end{equation}
Equality holds when $C=\mathcal M_n$.

To evaluate \eqref{eq:osc-VC}, let $p=(1+u)/2$ and, for
$Z_m\sim\operatorname{Bin}(m,p)$, define
\[
 M_m(p)=\mathbb P(Z_m>m/2)
        +\frac12\mathbb P(Z_m=m/2).
\]
Complementing all bits and noting that the tie probabilities
agree gives $\Delta_m(u)=2M_m(p)-1$.
For $r\ge1$, let $Z\sim\operatorname{Bin}(2r-1,p)$.
Conditioning on $Z$ after adding one or two independent
Bernoulli$(p)$ bits gives
\begin{align}
 M_{2r}(p)-M_{2r-1}(p)
 &=\frac p2\mathbb P(Z=r-1)
   -\frac{1-p}{2}\mathbb P(Z=r)=0,
   \notag\\
 M_{2r+1}(p)-M_{2r-1}(p)
 &=p^2\mathbb P(Z=r-1)
   -(1-p)^2\mathbb P(Z=r)
   \notag\\
 &=\binom{2r-1}{r-1}[p(1-p)]^r(2p-1).
 \label{eq:M-odd-diff}
\end{align}
Consequently, $\Delta_1(u)=u$ and
\[
 \begin{aligned}
 \Delta_{2r}(u)&=\Delta_{2r-1}(u),\\
 \Delta_{2r+1}(u)-\Delta_{2r-1}(u)
 &=\binom{2r}{r}4^{-r}u(1-u^2)^r.
 \end{aligned}
\]
For $I_r=\int_0^1(1-u^2)^r\,du$, integration by parts gives
$(2r+1)I_r=2rI_{r-1}$ for $r\ge1$.
Since $I_0=1$,
\[
 I_r=\frac{4^r}{(2r+1)\binom{2r}{r}}.
\]
Writing $m=\lceil n/2\rceil$ and summing the increments above,
we obtain
\begin{equation}\label{eq:optimal-oscillation}
 \begin{aligned}
 \operatorname{osc}V_C
 &\le n\int_0^1\Delta_n(u)\,\frac{du}{u}\\
 &=n\sum_{r=0}^{m-1}\binom{2r}{r}4^{-r}I_r
  =n\sum_{j=1}^{m}\frac1{2j-1}
  =D_n.
 \end{aligned}
\end{equation}

Since $A$ is increasing, any transition of $K$ from $z\in A$
to $y\notin A$ must be downward. Monotonicity of $V_C$ therefore
implies
\begin{equation}\label{eq:P-K-VC}
 (P-K)V_C(z)
 =\sum_{y\notin A}K(z,y)\bigl(V_C(z)-V_C(y)\bigr)
 \ge0.
\end{equation}
Combining \eqref{eq:P-K-VC} with \eqref{eq:free-poisson} gives
\begin{equation}\label{eq:constrained-drift}
 (P-I)V_C(z)\ge\mu(C)-\mathbf1_C(z),
 \qquad z\in A.
\end{equation}

Let $\mathcal F_t=\sigma(X_0,\ldots,X_t)$.
For every integer $M\ge1$, since
$\{t<T\wedge M\}\in\mathcal F_t$, summing the conditional
expected increments gives
\[
 \mathbb E_x[V_C(X_{T\wedge M})-V_C(x)]
 =\mathbb E_x\sum_{t=0}^{T\wedge M-1}(P-I)V_C(X_t).
\]
Since $V_C$ is bounded and $\mathbb E_xT<\infty$, we may
let $M\to\infty$ by dominated convergence: the potential
difference is bounded by $2\|V_C\|_\infty$, and
\[
 \left|\sum_{t=0}^{T\wedge M-1}(P-I)V_C(X_t)\right|
 \le 2\|V_C\|_\infty T.
\]
Thus
\begin{equation}\label{eq:V-martingale}
 \mathbb E_x[V_C(X_T)-V_C(x)]
 =\mathbb E_x\sum_{t=0}^{T-1}(P-I)V_C(X_t).
\end{equation}
Using \eqref{eq:constrained-drift} and
\eqref{eq:optimal-oscillation}, we conclude that
\[
 \begin{aligned}
 \mu(C)\mathbb E_xT
 -\mathbb E_x\sum_{t=0}^{T-1}\mathbf1_C(X_t)
 &\le \mathbb E_x[V_C(X_T)-V_C(x)]\\
 &\le\operatorname{osc}V_C\le D_n,
 \end{aligned}
\]
which proves \eqref{eq:occupation}.

Finally, take $A=\Omega_n$, $C=\mathcal M_n$, $x=0^n$, and $T=\inf\{t\ge0:X_t=1^n\}.$ Finite state irreducibility gives $\mathbb E_{0^n}T<\infty$.
Since $P=K$, the drift relation is an equality. Hence
\eqref{eq:V-martingale}, together with equality in
\eqref{eq:osc-VC} and the evaluation in
\eqref{eq:optimal-oscillation}, gives
\[
 \begin{aligned}
 \mu(C)\mathbb E_{0^n}T
 -\mathbb E_{0^n}\sum_{t=0}^{T-1}\mathbf1_C(X_t)
 &=V_C(1^n)-V_C(0^n)\\
 &=D_n.
 \end{aligned}
\]
Thus $D_n$ is attained and cannot be decreased.
\end{proof}

For $H_k=\sum_{j=1}^k j^{-1}$, the identity
$\sum_{j=1}^m(2j-1)^{-1}=H_{2m}-\frac12H_m$, together with
$H_k=\log k+\gamma_{\mathrm E}+O(k^{-1})$ and $m=n/2+O(1)$, gives
\[
 D_n=n\left(H_{2m}-\frac12H_m\right)
 =\frac12 n\log n+\frac{\log2+\gamma_{\mathrm E}}{2}\,n+O(1),
\]
where $\gamma_{\mathrm E}$ is Euler's constant.

\section{The killed Green function and its truncations}
\label{sec:green}
The Green function records expected visits before hitting a target.
It therefore turns the stopping-time occupation inequality into
an integral inequality, to which the local functional estimate
can be applied. Fix an arbitrary set $B\subseteq A$ satisfying
$\pi(B)\ge1/4$.
Since $\tau_B=0$ when the initial state lies in $B$, it suffices to
consider $x\in D:=A\setminus B$. For the remainder of this section,
fix such an $x$. 

Since $P$ is irreducible on the finite state space $A$ and
$B\neq\varnothing$, we have
$\max_{z\in A}\mathbb E_z\tau_B<\infty$. Let $P_D=(P(y,z))_{y,z\in D}$. The Markov property gives
\begin{equation}\label{eq:killed-transition}
 P_D^t(y,z)=\mathbb P_y(X_t=z,\ t<\tau_B).
\end{equation}
Thus the Green matrix series of the killed chain converges, and
\begin{equation}\label{eq:resolvent}
 \mathsf Z_D=\sum_{t\ge0}P_D^t=(I-P_D)^{-1}.
\end{equation}

Define the Green density with respect to $\pi$ by
\begin{equation}
 G(y)
 =
 \frac1{\pi(y)}
 \E_x\sum_{t=0}^{\tau_B-1}\one_{\{X_t=y\}},
 \qquad y\in A.
 \label{eq:green-definition}
\end{equation}
It is nonnegative, finite, and vanishes on $B$. Since $\mu(y)= a\pi(y), y\in A,$ Tonelli's theorem gives 
\begin{equation}
 \int_A G\,\dd\mu
 =
  a\E_x\tau_B,
 \quad \text{and} \quad
 \int_C G\,\dd\mu
 =
  a\E_x\sum_{t=0}^{\tau_B-1}\one_C(X_t).
 \label{eq:green-integrals}
\end{equation}
Moreover, by \eqref{eq:killed-transition}, 
\begin{equation}
 G(y)
 =
 |A|\,\mathsf Z_D(x,y),
 \qquad y\in D.
 \label{eq:green-matrix}
\end{equation}

Now we can present the Dirichlet spectral gap of the killed chain.

\begin{lemma}
\label{lem:dirichlet-boundary}

Every $f:A\to\R$ satisfying $f|_B=0$ obeys
\begin{equation}
 \cE_A(f)
 \ge
 \frac{a}{8n}\norm{f}{2,\pi}^2.
 \label{eq:dirichlet-boundary}
\end{equation}
Consequently, on $\ell^2(D,\pi|_D)$,
\begin{equation}
 \norm{\mathsf Z_D}{2\to2}
 \le
 \frac{8n}{a},
 \label{eq:resolvent-norm}
\end{equation}
and, for every $y\in A$,
\begin{equation}
 0\le G(y)\le8n2^n.
 \label{eq:green-max}
\end{equation}

\end{lemma}

\begin{proof}

Since $f$ vanishes on $B$, Cauchy--Schwarz gives
\[
 \pi(f)^2
 =
 \left(\sum_{y\in D}\pi(y)f(y)\right)^2
 \le
 \pi(D)\sum_{y\in D}\pi(y)f(y)^2.
\]
Hence
\begin{equation}
\Var_\pi(f)
\ge
(1-\pi(D))\pi(f^2)
=
\pi(B)\norm{f}{2,\pi}^2
\ge
\frac14\norm{f}{2,\pi}^2.
\label{eq:var_bound}
\end{equation}
Combining \eqref{eq:var_bound} with \eqref{eq:global-poincare} proves
\eqref{eq:dirichlet-boundary}.

Now let a function on $D$ be extended by zero to $B$. Then
\[
 \ip{f}{(I-P_D)f}{\pi|_D}
 =
 \ip{f}{(I-P)f}{\pi}
 =
 \cE_A(f).
\]
Since $P$ is reversible with respect to $\pi$, the principal
submatrix $P_D$ is self-adjoint on $\ell^2(D,\pi|_D)$. Therefore
\eqref{eq:dirichlet-boundary} implies that the smallest eigenvalue of
$I-P_D$ is at least $ a/(8n)$. By \eqref{eq:resolvent}, the eigenvalues of $\mathsf Z_D$ are the reciprocals of those of
$I-P_D$. As $\mathsf Z_D$ is self-adjoint,
\[
\norm{\mathsf Z_D}{2\to2}
=
\frac{1}{\lambda_{\min}(I-P_D)}
\le
\frac{8n}{a},
\]
which proves \eqref{eq:resolvent-norm}.

Finally, since $\pi$ is uniform on $A$, the measure $\pi|_D$ is a
constant multiple of counting measure. Hence, for every
$y\in D$,
\[
 0\le
 \mathsf Z_D(x,y)
 \le
 \norm{\mathsf Z_D}{2\to2}
 \le
 \frac{8n}{a}.
\]
Using \eqref{eq:green-matrix} and $|A|=2^n a,$ we obtain
\[
 0\le G(y)
 \le
 |A|\frac{8n}{a}
 =
 8n2^n,
 \qquad y\in D.
\]
The same bound holds on $B$, where $G=0$.

\end{proof}

The Green density satisfies the following energy identity.

\begin{lemma}
\label{lem:green-source}

For every $f:A\to\R$ satisfying $f|_B=0$,
\begin{equation}
 \cE_A(G,f)=f(x).
 \label{eq:green-source}
\end{equation}
\end{lemma}

\begin{proof}
The restriction $P_D$ remains reversible with respect to $\pi|_D$,
and hence so does every $P_D^t$. Therefore, for $y\in D$,
\[
 G(y)
 =
 \frac{\mathsf Z_D(x,y)}{\pi(y)}
 =
 \frac{\mathsf Z_D(y,x)}{\pi(x)}.
\]
Using \eqref{eq:resolvent},
\[
\begin{aligned}
 (I-P_D)G(y)
 &=
 \frac{1}{\pi(x)}
 \bigl[(I-P_D)\mathsf Z_D\bigr](y,x)\\
 &=
 \frac{\one_{\{y=x\}}}{\pi(x)}.
\end{aligned}
\]
Since $G=0$ on $B$,
\[
 (I-P)G(y)=(I-P_D)G(y),
 \qquad y\in D.
\]

By reversibility and the assumption $f|_B=0$,
\begin{align*}
 \cE_A(G,f)
 &=
 \ip{(I-P)G}{f}{\pi}\\
 &=
 \sum_{y\in D}
 \pi(y)\frac{\one_{\{y=x\}}}{\pi(x)}f(y)
 =
 f(x).
\end{align*}

\end{proof}

\begin{lemma}
\label{lem:nonlinear-green}

Let $\psi:[0,\infty)\to[0,\infty)$ be nondecreasing and
$1$-Lipschitz, with $\psi(0)=0$. Then
\begin{equation}
 \cE_A(\psi(G))
 \le
 \psi(G(x)).
 \label{eq:nonlinear-green}
\end{equation}

For every $s>0$,
\begin{equation}
 \mu\{y\in A:G(y)>s\}
 \le
 \frac{8n}{s}.
 \label{eq:green-level-tail}
\end{equation}
Moreover, for $f_s=\min\{(G-s)_+,s\},$
we have $\cE_A(f_s)\le s,$ and 
\begin{equation}
 \supp\widetilde f_s
 \subseteq
 \{y\in A:G(y)>s\}.
 \label{eq:truncation-properties}
\end{equation}

\end{lemma}

\begin{proof}
For all $u,v\ge0$, monotonicity and the $1$-Lipschitz property give
\begin{equation}
(\psi(u)-\psi(v))^2
\le
(u-v)(\psi(u)-\psi(v)).
\label{eq:psi_ineq}
\end{equation}
Applying \eqref{eq:psi_ineq} to $G(y)$ and $G(z)$ and summing against the
nonnegative coefficients in \eqref{eq:dirichlet} yields
\[
 \cE_A(\psi(G))
 \le
 \cE_A(G,\psi(G)).
\]
Since $G=0$ on $B$ and $\psi(0)=0$, the function $\psi(G)$ also
vanishes on $B$. Lemma~\ref{lem:green-source} therefore gives
\[
 \cE_A(\psi(G))
 \le
 \psi(G(x)).
\]

Now take $\psi(u)=\min\{u,s\}$. Since
$\min\{G,s\}$ vanishes on $B$, Lemma~\ref{lem:dirichlet-boundary}
and \eqref{eq:nonlinear-green} give
\begin{align*}
 \int_A\min\{G,s\}^2\,\dd\mu
 &=
  a\norm{\min\{G,s\}}{2,\pi}^2\\
 &\le
 8n\cE_A(\min\{G,s\})\\
 &\le
 8ns.
\end{align*}
On $\{y\in A:G(y)>s\}$, the squared truncation equals $s^2$.
Hence
\[
 s^2\mu\{y\in A:G(y)>s\}
 \le
 8ns.
\]

Finally, the function $u\longmapsto\min\{(u-s)_+,s\}$ is nondecreasing, $1$-Lipschitz, vanishes at $0$, and is bounded by
$s$. Applying \eqref{eq:nonlinear-green} gives $\cE_A(f_s)
 \le
 f_s(x)
 \le
 s.$ Moreover, $f_s(y)=0$ whenever $G(y)\le s$, which gives
\[
 \supp\widetilde f_s
 \subseteq
 \{y\in A:G(y)>s\}.
\]

\end{proof}

\section{Proof of the main result}\label{sec:main-proof}
\label{sec:summation}

\subsection{The hitting time bound}

Recall that $\delta= a/8$ and $\mu(C)\ge a/2$.
Fix $B\subseteq A$ with $\pi(B)\ge1/4$ and $x\in A\setminus B$,
and let $G$ and $f_s$ be as in Section~5. We first choose the
hypercontractive parameter at each dyadic level so that the error
coefficient is independent of the level.

\begin{lemma}
\label{lem:level-inequality}
For every $s\ge8en$ and every integer $j\ge0$,
\begin{equation}
 \frac34\,\frac{\displaystyle\int_C f_{2^j s}^{\,2}\dd\mu}{2^j s}
 \le
 \frac{32n\log(s/(8n))}{2\log(s/(8n))+j\log2}
 +\frac{8n}{\delta s}\,
       \frac{\displaystyle\int_A f_{2^j s}^{\,2}\dd\mu}{2^j s}.
 \label{eq:level-inequality}
\end{equation}
\end{lemma}

\begin{proof}
By Lemma~\ref{lem:nonlinear-green}, $\mu(\supp\widetilde f_{2^j s})\le\frac{8n}{2^j s},$ and $\cE_A(f_{2^j s})\le2^j s.$ Choose
\[
 \theta=\frac{2\log(s/(8n))}{2\log(s/(8n))+j\log2}\in(0,1],
 \qquad
 \left(\frac{8n}{2^j s}\right)^{\theta/(2-\theta)}=\frac{8n}{s}.
\]
Lemma~\ref{lem:support-factor} therefore gives
\[
 \|\widetilde f_{2^j s}\|_{2-\theta,\mu}^{\,2}
 \le\frac{8n}{s}\int_A f_{2^j s}^{\,2}\dd\mu.
\]
Substituting into Proposition~\ref{prop:core-functional}
and using $2 an/\delta=16n$, we obtain
\[
 \frac34\int_C f_{2^j s}^{\,2}\dd\mu
 \le16n\theta\,2^j s
       +\frac{8n}{\delta s}\int_A f_{2^j s}^{\,2}\dd\mu.
\]
Division by $2^j s$ proves the claim.
\end{proof}

The next Lemma explains the choice of truncations:
their normalized squares sum to a quantity comparable to $G$.
Thus summing the quadratic estimates recovers expected occupation.

\begin{lemma}
\label{lem:scalar-summation}
For every $u\ge0$,
\begin{equation}
 \frac u2-1
 \le\sum_{j\ge0}2^{-j}\min\{(u-2^j)_+,2^j\}^{\,2}
 \le u.
 \label{eq:scalar-summation}
\end{equation}
Only finitely many terms are nonzero.
\end{lemma}

\begin{proof}
The assertion is immediate for $u<1$. If $2^k\le u<2^{k+1}$,
then
\begin{equation}
 \sum_{j\ge0}2^{-j}\min\{(u-2^j)_+,2^j\}^{\,2}
 =2^k-1+\frac{(u-2^k)^2}{2^k}.
 \label{eq:scalar-exact}
\end{equation}
The right-hand side is at least $2^k-1\ge u/2-1$ and at most
$2^k-1+(u-2^k)=u-1\le u$.
\end{proof}

\begin{proof}[Proof of \eqref{eq:main-hitting}]
Let $s\ge8en$. By \eqref{eq:green-max},
\begin{equation}
 G(y)\le8n2^n<2^n s,\qquad y\in A,
 \label{eq:finite-level-cutoff}
\end{equation}
so $f_{2^j s}=0$ for $j\ge n$. Applying
Lemma~\ref{lem:scalar-summation} to $G(y)/s$ gives
\[
 \frac{G(y)}2-s
 \le\sum_{j=0}^{n}\frac{f_{2^j s}(y)^2}{2^j s}
 \le G(y).
\]
An integral comparison yields
\begin{align}
 \sum_{j=0}^{n}
 \frac{\log(s/(8n))}{2\log(s/(8n))+j\log2}
 &\le\frac12+\frac{\log(s/(8n))}{\log2}
       \log\left(1+\frac{n\log2}{2\log(s/(8n))}\right)
       \notag\\
 &\le2\log(s/(8n))\log(en),
 \label{eq:dyadic-sum-bound}
\end{align}
where the last inequality uses $\log(s/(8n))\ge1$.
Summing \eqref{eq:level-inequality} over $j=0,\ldots,n$ and
using the preceding estimates, we obtain
\begin{equation}
 \frac34\left(\frac12\int_C G\dd\mu-\mu(C)s\right)
 \le64n\log\frac{s}{8n}\log(en)
       +\frac{8n}{\delta s}\int_A G\dd\mu.
 \label{eq:integrated-summation}
\end{equation}

Choose
\[
 s=\frac{1024n}{3 a\mu(C)}>8en,
 \qquad
 \frac{8n}{\delta s}=\frac{3\mu(C)}{16}.
\]
Rearranging \eqref{eq:integrated-summation} gives
\begin{equation}
 \begin{split}
 \int_C G\dd\mu
 &\le\frac{\mu(C)}2\int_A G\dd\mu
       +\frac{512n}{3}
         \log\frac{128}{3 a\mu(C)}\log(en)\\
 &\qquad+\frac{2048n}{3 a}.
 \end{split}
 \label{eq:occupation-upper}
\end{equation}
Applying
Proposition~\ref{prop:occupation} with $T=\tau_B$ and using
\eqref{eq:green-integrals}, we obtain
\begin{equation}
 \mu(C)\int_A G\dd\mu
 \le\int_C G\dd\mu+ aD_n.
 \label{eq:occupation-lower}
\end{equation}
Combining \eqref{eq:occupation-upper} and
\eqref{eq:occupation-lower}, absorbing the term
$\frac12\mu(C)\int_A G\dd\mu$, and using
$\int_A G\dd\mu= a\mathbb E_x\tau_B$, we conclude that
\begin{equation}
 \begin{split}
 \mathbb E_x\tau_B
 &\le\frac{1024n}{3 a\mu(C)}
       \log\frac{128}{3 a\mu(C)}\log(en)\\
 &\qquad+\frac{4096n}{3 a^2\mu(C)}
       +\frac{2D_n}{\mu(C)}.
 \end{split}
 \label{eq:H-before-constants}
\end{equation}
%The error has been summed over all of $A$ before absorption;
%no comparison between individual truncations on $C$ and on
%$A\setminus C$ is needed.

Since $\mu(C)\ge a/2$,
\begin{equation}
 \log\frac{128}{3 a\mu(C)}
 \le\log\frac{256}{3 a^2}
 \le5\log\frac e{a}.
 \label{eq:log-density-bound}
\end{equation}
Together with $D_n\le n\log(en)$, this gives
\begin{equation}
 \begin{split}
 \mathbb E_x\tau_B
 &\le\frac{10240}{3 a^2}
       n\log\frac e{a}\log(en)\\
 &\qquad+\frac{8192}{3 a^3}n
       +\frac4{a}n\log(en).
 \end{split}
 \label{eq:H-refined}
\end{equation}
%Estimate \eqref{eq:H-refined} is stronger than the convenient
%$a^{-3}$ coarsening in \eqref{eq:main-hitting}. 
Using
$a\log(e/a)\le1$, $ a\le1$, and
$\log(en)\ge1$, we obtain
\begin{equation}
 \mathbb E_x\tau_B
 \le\left(\frac{10240+8192}{3}+4\right)
        a^{-3}n\log(en)
 =6148\, a^{-3}n\log(en).
 \label{eq:pointwise-large-hit}
\end{equation}
This bound is uniform in $B$ and $x\in A\setminus B$, and is
trivial for $x\in B$. Taking the maximum proves
\eqref{eq:main-hitting}.
\end{proof}

\subsection{From hitting to mixing}
\label{sec:mixing}

The singleton case is immediate. Otherwise, since $P$ is lazy,
$2P-I$ is an irreducible reversible Markov kernel whose lazy
version is $P$. Inserting independent geometric waiting times
of mean $2$ before its transitions gives the $P$-chain. Thus
\begin{equation}
 \mathbb E_x^{2P-I}\tau_B=\frac12\mathbb E_x\tau_B,
 \qquad
 \thit^{2P-I}(1/4)=\frac12\thit^P(1/4).
 \label{eq:hitting-clock-relation}
\end{equation}
Applying Theorem~\ref{thm:PS} to $2P-I$ with $\alpha=1/4$ gives
\[
 \tmix(P)
 \le C_{1/4}\thit^{2P-I}(1/4)
 =\frac{C_{1/4}}2\thit^P(1/4)
 \le3074C_{1/4}\, a^{-3}n\log(en).
\]
This proves the Question~\ref{question:DM}. 
Theorem~\ref{chenapp:thm:mixing} in the appendix sharpens this to
\eqref{eq:main-mixing}.

\section{Density dependence and sharpness}\label{sec:concluding}

The order $n\log n$ at fixed density is sharp for every increasing
set, in the following uniform sense. Write
$d_A(t)=\max_{x\in A}\|P^t(x,\cdot)-\pi\|_{\mathrm{TV}}$ and
$t_{\mathrm{mix}}^A(\varepsilon)=\min\{t:d_A(t)\le\varepsilon\}$.

\begin{proposition}\label{prop:universal-lower}
For $n\ge2$, integer $t\ge0$, and $q_t=(1-1/n)^t$,
\begin{equation}\label{eq:universal-lower}
 d_A(t)\ge1-\frac4{nq_t^2}
              -\frac1a\exp\left(-\frac{nq_t^2}{8}\right).
\end{equation}
Consequently, for fixed $c>0$ and $\varepsilon\in(0,1)$,
uniformly over increasing sets with $a\ge c$,
\begin{equation}\label{eq:mix-time-lower}
 t_{\mathrm{mix}}^A(\varepsilon)
 \ge\frac12 n\log n-O_{c,\varepsilon}(n).
\end{equation}
\end{proposition}

\begin{proof}
Start the censored walk $X$ and the free walk $Y$ from $1^n$ and
use the same coordinate choices and refreshed bits. Then
$X_t\ge Y_t$: a rejected refresh can only be downward.
%This comparison does not require monotonicity of the censored
%kernel between two censored chains.
For $S(y)=\sum_{j=1}^n(2y_j-1)$, we have $\mathbb ES(Y_t)=nq_t$, $\operatorname{Var}(S(Y_t))\le n.$ Indeed, distinct spins have product expectation $(1-2/n)^t\le q_t^2$
and hence nonpositive covariance. At $t=0$ the product is one.
For $F_t=\{S\ge nq_t/2\}$, Chebyshev and the coupling give
$\mathbb P(X_t\in F_t)\ge1-4/(nq_t^2)$.
Under stationarity, Hoeffding's inequality gives
$\pi(F_t)\le a^{-1}\mu(F_t)\le a^{-1}e^{-nq_t^2/8}$.
Subtracting proves \eqref{eq:universal-lower}. At
$t=\lfloor\frac12n\log n-\kappa n\rfloor$, one has $nq_t^2\to e^{2\kappa}$.
Choosing $\kappa$ sufficiently large in terms of $c,\varepsilon$ proves \eqref{eq:mix-time-lower}.
\end{proof}

%\clearpage
\appendix
\section{Appendix: Dependence of the mixing time on $\mu(A)$}
\label{app:chen-improved-density}
\begin{center}
{\large Yuval Peres}\\[2mm]
Beijing Institute of Mathematical Sciences and Applications (BIMSA), Beijing, China
\end{center}
\theoremstyle{plain}
\newtheorem{chenappthm}{Theorem}[section]
\newtheorem{chenapplemma}[chenappthm]{Lemma}
\newtheorem{chenappprop}[chenappthm]{Proposition}
\numberwithin{equation}{section}
\setcounter{equation}{0}

The appendix proves $C_a=13a^{-1}\log(e/a)$.
Cohen's conjectured bound $O(n\log n/a)$~\cite[p.~36]{Cohen}
would remove the remaining logarithmic factor in the density.

\begin{chenappthm}\label{chenapp:thm:mixing}
For every nonempty increasing $A\subseteq\Omega_n$,
\begin{equation}\label{chenapp:eq:mixing-constant}
 t_{\mathrm{mix}}(P)\le13\frac n a\log\frac ea\log(en).
\end{equation}
The mixing threshold is $1/4$, as in \eqref{eq:tmix}. 
%The constant
%$13$ is an admissible explicit constant; no optimality is claimed.
\end{chenappthm}

The proof has two parts: a bound for hitting sets of stationary mass at
least $15/16$, followed by an explicit passage from hitting to mixing.
The singleton case is immediate. Henceforth $|A|\ge2$, and we write $\gamma=1-\sqrt{1-a},$ and $F_\rho=T_\rho\mathbf1_A.$ For $y\in A$, $F_\rho(y)$ is the probability that a cube-noise
proposal from $y$ remains in $A$. We use these acceptance probabilities
in place of the fixed core from the main proof, allowing the noise
level to vary with the level of the Green function.
The spectral gap is at least $\gamma/n$ by Theorem~\ref{thm:FFP}.
We also use the occupation constant
$D_n\le n(1+\frac12\log n)$ from \eqref{eq:coupon-upper}.

\begin{chenapplemma}\label{chenapp:lem:noise}
Let $f:A\to\mathbb R$, and
set $r:=\mu\bigl(\{x\in A:f(x)\ne0\}\bigr).$ Then
\begin{equation}\label{chenapp:eq:noise}
 \pi(F_{1-\theta}f^2)
 \le2n\theta\mathcal E_A(f)
       +r^{\theta/(2-\theta)}\pi(f^2),\qquad 0<\theta\le1.
\end{equation}
\end{chenapplemma}

\begin{proof}
Choose the refreshed coordinates $S$ independently with
probability $\theta$, and average over the uniformly frozen
coordinates $z$. In each nonempty section,
\begin{equation}\label{eq:var-identity}
 (\mu_{S,z}(A))^2\operatorname{Var}_{\pi_{S,z}}f
 =\mu_{S,z}(A)\mu_{S,z}(\widetilde f^2)
       -\mu_{S,z}(\widetilde f)^2.
\end{equation}
Average \eqref{eq:var-identity}, using Lemma~\ref{lem:mean-averaging} for the
squared means. Then apply Lemmas~\ref{lem:section-poincare}
and~\ref{lem:energy-averaging} to obtain
\[
 \pi(F_{1-\theta}f^2)
 -a^{-1}\langle\widetilde f,T_{1-\theta}\widetilde f\rangle_\mu
 =a^{-1}\mathbb E_{S,z}\!\left[
    \mu_{S,z}(A)^2\operatorname{Var}_{\pi_{S,z}}f\right]
 \le2n\theta\mathcal E_A(f).
\]
Empty sections contribute zero, the last inequality averages the
section energies, with each coordinate refreshed with probability
$\theta$. The noise semigroup identity $T_\rho T_\sigma=T_{\rho\sigma}$, hypercontractivity \eqref{eq:hypercontractivity}, and
Lemma~\ref{lem:support-factor} give
\[
 \langle\widetilde f,T_{1-\theta}\widetilde f\rangle_\mu
 =\|T_{\sqrt{1-\theta}}\widetilde f\|_{2,\mu}^2
 \le\|\widetilde f\|_{2-\theta,\mu}^2
 \le r^{\theta/(2-\theta)}a\pi(f^2).
\]

\end{proof}

The weight $F_\rho$ need not be bounded below by $a$ at every point.
The next lemma bounds its deficit through increasing potentials.

\begin{chenapplemma}\label{chenapp:lem:potential}
For $0\le\rho<1$, let $h_\rho=n\int_0^\rho(F_u-a)\,\frac{du}{u}.$ Then on $A$, $(P-I)h_\rho\ge a-F_\rho$, and
\begin{equation}\label{chenapp:eq:potentials}
\sum_{j\ge1}\operatorname{osc}_A(h_{\rho_j}-h_{\rho_{j-1}})
 \le D_n\quad(0=\rho_0\le\rho_1\le\cdots<1),
\end{equation}
where $\operatorname{osc}_A h=\max_Ah-\min_Ah$.
\end{chenapplemma}

\begin{proof}
The integral is finite by the Walsh expansion, as in
\eqref{eq:occupation-potential}. The identity $(K-I)h_\rho=a-F_\rho$ follows by integrating
$(K-I)T_u=-(u/n)\partial_uT_u$. Every $h_\rho$ is increasing,
so censoring downward moves increases its drift. Each potential
increment is also increasing, and its oscillation is at most
$n\int_{\rho_{j-1}}^{\rho_j}\Delta_n(u)\,du/u$.
Summing and using \eqref{eq:optimal-oscillation} proves \eqref{chenapp:eq:potentials}.

\end{proof}

We return to the killed Green density $G$ defined in
\eqref{eq:green-definition}. The arguments establishing
\eqref{eq:green-integrals} and \eqref{eq:green-source}
apply to any nonempty target $B\subseteq A$.
Thus, for every $x\in A\setminus B$, $\pi(G)=\mathbb E_x\tau_B,$ and for every $f:A\to\mathbb R$ with $f|_B=0$,
\[
\mathcal E_A(G,f)=f(x).
\]

To apply \eqref{chenapp:eq:noise} at successive Green function levels, fix \(q=e^4\). For \(t>0\), define the truncation functions on \([0,\infty)\) by
\[
\psi_t(u)=
\begin{cases}
0,&0\le u\le t,\\[2pt]
\dfrac{(u-t)^2}{u},&u>t,
\end{cases}
\]
and then define $\phi_t(u)=\psi_t(u)-\psi_{qt}(u)$, $v_t(u)=\sqrt{t\phi_t(u)}.$

The function \(\phi_t\) is nonnegative and vanishes on \([0,t]\). For every \(s>0\), integer \(J\ge1\), and \(u\ge0\),
\[
\sum_{j=0}^{J-1}\frac{v_{q^js}(u)^2}{q^js}
=
\sum_{j=0}^{J-1}\phi_{q^js}(u)
=
\psi_s(u)-\psi_{q^Js}(u).
\]

The next lemma provides the level set and energy bounds for \(v_t(G)\), together with a bound on \(\mathcal E_A(h,\psi_t(G))\) in terms of $\operatorname{osc}_A h$.
\begin{chenapplemma}\label{chenapp:lem:truncations} Let \(B\subseteq A\) satisfy \(\pi(B)\ge\beta>0\), fix \(x\in A\setminus B\), and let \(G\) be the killed Green density. Then $\max G\le\frac{n|A|}{\beta\gamma},$ and 
\begin{equation}\label{chenapp:eq:green}
 \begin{gathered}
\mu\bigl(\{y\in A:G(y)>t\}\bigr)
\le\frac{an}{\beta\gamma\,t}
\quad(t>0).
 \end{gathered}
\end{equation}

Moreover, for every \(t>0\) and every real-valued function \(h\) on \(A\), $ \mathcal E_A(v_t(G))\le t(1+\log2),$ and 

\begin{equation}\label{chenapp:eq:truncations}
 |\mathcal E_A(h,\psi_t(G))|\le\operatorname{osc}_A h.
\end{equation}

\end{chenapplemma}

\begin{proof}
For every \(f:A\to\mathbb R\) with \(f|_B=0\), Cauchy–Schwarz gives $\pi(f)^2
\le \pi(A\setminus B)\pi(f^2),$ and 
\[
\operatorname{Var}_\pi(f)
\ge\pi(B)\pi(f^2)
\ge\beta\pi(f^2).
\]
Theorem~\ref{thm:FFP} therefore implies
\[
\mathcal E_A(f)
\ge\frac{\gamma}{n}\operatorname{Var}_\pi(f)
\ge\frac{\beta\gamma}{n}\pi(f^2).
\]
The killed resolvent consequently has norm at most \(n/(\beta\gamma)\). Since $\pi$ is uniform, \eqref{eq:green-matrix} gives the maximum bound.
The truncation argument in Lemma~\ref{lem:nonlinear-green} also gives
$\mathcal E_A(G\wedge t)\le t$, whence
\[
 t^2\mu\{G>t\}\le a\pi((G\wedge t)^2)
 \le\frac{an}{\beta\gamma}\mathcal E_A(G\wedge t)
 \le\frac{ant}{\beta\gamma}.
\]

For locally absolutely continuous $v$ with
$v(0)=0$ and square integrable derivative, set
$w(u)=\int_0^u|v'(r)|^2\,dr$. Cauchy--Schwarz gives
$(v(b)-v(c))^2\le(b-c)(w(b)-w(c))$ for $b\ge c\ge0$.
Applying this to each edge and using \eqref{eq:green-source} yields
\[
 \mathcal E_A(v(G))\le\mathcal E_A(G,w(G))
 =w(G(x))\le\int_0^\infty|v'(r)|^2\,dr.
\]
Writing $v_t(u)=tV(u/t)$ gives
\[
 V(r)=\begin{cases}
 0,&r\le1,\\
 (r-1)/\sqrt r,&1<r\le q,\\
 \sqrt{(q-1)(2-(q+1)/r)},&r>q.
 \end{cases}
\]
Differentiating and integrating the two nonzero expressions yields
\begin{align*}
 \int_1^q|V'(r)|^2\,dr
 &=\frac14\log q+\frac58-\frac1{2q}-\frac1{8q^2},\\
 \int_q^\infty|V'(r)|^2\,dr
 &=\frac{q-1}{q+1}\log\frac{2q}{q-1}
   -\frac58+\frac1{2q}+\frac1{8q^2}.
\end{align*}
Thus
\[
 \int_0^\infty|V'(r)|^2\,dr
 =\frac14\log q+\frac{q-1}{q+1}\log\frac{2q}{q-1}
 \le1+\log2.
\]
The last inequality uses $q=e^4$ and
$z\log(1+1/z)\le\log2$ for $0<z\le1$, with
$z=(q-1)/(q+1)$. Scaling back to $v_t$ gives
%The last inequality is based on the fact that
%$z\log(1+1/z)\le\log2$ for $0<z\le1$.
%Since $\int_0^\infty|v_t'(u)|^2\,du
%=t\int_0^\infty|V'(r)|^2\,dr$, 
\[
\mathcal E_A(v_t(G))
\le\int_0^\infty|v_t'(u)|^2\,du
=t\int_0^\infty|V'(r)|^2\,dr
\le t(1+\log2).
\]

Now apply the Green source identity ~\eqref{eq:green-source}  with \(f=\mathbf1_{\{y\}}\) to obtain
\[
(P-I)G(y)=-\frac{\mathbf1_{\{y=x\}}}{\pi(x)},
\qquad y\in A\setminus B.
\]
On \(B\), we have \(G=0\) and \(\psi_t'(G)=0\). Since \(\psi_t\) is convex and \(0\le\psi_t'\le1\),
\[
\begin{aligned}
\bigl[(P-I)\psi_t(G)\bigr](y)
&\ge \psi_t'(G(y))(P-I)G(y)\\
&\ge-\frac{\mathbf1_{\{y=x\}}}{\pi(x)},
\qquad y\in A.
\end{aligned}
\]

In particular, \((P-I)\psi_t(G)\) is nonnegative on \(A\setminus\{x\}\). Stationarity therefore gives
\[
0\le
\sum_{y\in A\setminus\{x\}}
\pi(y)\bigl[(P-I)\psi_t(G)\bigr](y)
=
-\pi(x)\bigl[(P-I)\psi_t(G)\bigr](x)
\le1.
\]

Using \eqref{eq:dirichlet} and \(\pi((P-I)\psi_t(G))=0\), we conclude that
\[
\begin{aligned}
\left|\mathcal E_A(h,\psi_t(G))\right|
&=
\left|
\sum_{y\in A\setminus\{x\}}
\pi(y)\bigl(h(y)-h(x)\bigr)
\bigl[(P-I)\psi_t(G)\bigr](y)
\right|\\
&\le
\operatorname{osc}_A h
\sum_{y\in A\setminus\{x\}}
\pi(y)\bigl[(P-I)\psi_t(G)\bigr](y)\\
&\le \operatorname{osc}_A h.
\end{aligned}
\]
\end{proof}

We now choose the scales so that the support term in
\eqref{chenapp:eq:noise} is always $a/e$. The successive potential
increments will together contribute at most $D_n$. The target mass
$15/16$ is chosen to match the explicit hitting-to-mixing estimate
in Lemma~\ref{chenapp:lem:direct-tv}.

\begin{chenappprop}\label{chenapp:prop:hitting}
\begin{equation}\label{chenapp:eq:hitting}
 \max_{x,B:\,\pi(B)\ge15/16}\mathbb E_x\tau_B
 \le\frac{32en}{15\gamma}
 +\frac{en}{(e-1)a}
       \left[(1+\log2)(\Xi+2)+\frac{D_n}{n}\right],
\end{equation}
where $\Xi=\ell\log(1+n\log2/(2\ell))$, and $\ell=\log(e/a)$.
\end{chenappprop}

\begin{proof}
Fix a target $B$ with $\pi(B)\ge15/16$ and a starting point
$x\notin B$, and apply Lemma~\ref{chenapp:lem:truncations} to its
Green density $G$. Let $s=\frac{16en}{15\gamma},$ $\theta_j=\frac{\ell}{\ell+2j},$ and $J=\left\lceil\frac{n\log2}{4}\right\rceil$. By \eqref{chenapp:eq:green}, $q^Js>\max G$ and
$\mu\{G>q^js\}\le(a/e)q^{-j}$.
Apply \eqref{chenapp:eq:noise} to $v_{q^js}(G)$, whose square is
$q^js\phi_{q^js}(G)$, and divide by $q^js$. Since
$[(a/e)q^{-j}]^{\theta_j/(2-\theta_j)}=a/e$, we obtain
\begin{equation}\label{eq:pi-bound}
 \pi(F_{1-\theta_j}\phi_{q^js}(G))
 \le2n(1+\log2)\theta_j+(a/e)\pi(\phi_{q^js}(G)).
\end{equation}
On the other hand, multiply the $(P-I)h_\rho\ge a-F_\rho$ by $\phi_{q^js}(G)\ge0$
and integrate to get
\begin{equation}\label{eq:pi-lower-bound}
 \pi(F_{1-\theta_j}\phi_{q^js}(G))
 \ge a\pi(\phi_{q^js}(G))
 +\mathcal E_A(h_{1-\theta_j},\phi_{q^js}(G)).
\end{equation}
The terms
$\phi_{q^js}(G)=\psi_{q^js}(G)-\psi_{q^{j+1}s}(G)$
telescope.
Moreover, $h_{1-\theta_0}=0$ and $\psi_{q^Js}(G)=0$,
so summation by parts gives
\[
 \sum_{j=0}^{J-1}
 \mathcal E_A(h_{1-\theta_j},\phi_{q^js}(G))
 =
 \sum_{j=1}^{J-1}
 \mathcal E_A
 \bigl(h_{1-\theta_j}-h_{1-\theta_{j-1}},\psi_{q^js}(G)\bigr).
\]
Combining \eqref{eq:pi-bound} and \eqref{eq:pi-lower-bound}, and summing therefore yields
\begin{align*}
 a(1-e^{-1})\pi(\psi_s(G))
 &\le2n(1+\log2)\sum_{j<J}\theta_j\\
 &\quad
 -\sum_{j=1}^{J-1}
 \mathcal E_A
 \bigl(h_{1-\theta_j}-h_{1-\theta_{j-1}},\psi_{q^js}(G)\bigr)\\
 &\le2n(1+\log2)\sum_{j<J}\theta_j+D_n.
\end{align*}
The final inequality uses both \eqref{chenapp:eq:truncations}
and \eqref{chenapp:eq:potentials}.
Since $J-1<n\log2/4$, integration of the decreasing function
$u\mapsto\ell/(\ell+2u)$ gives
$\sum_{j<J}\theta_j\le1+\Xi/2$.
Finally, $\psi_s(u)\ge u-2s$ and $\pi(G)=\mathbb E_x\tau_B$.
Rearranging gives \eqref{chenapp:eq:hitting}.
\end{proof}

\begin{chenapplemma}\label{chenapp:lem:direct-tv}
Let $Q$ be a finite irreducible lazy reversible chain, with stationary
measure $\nu$ and spectral gap $\lambda>0$. If $\mathcal H\ge1$ and
$\max_{x,B:\,\nu(B)\ge15/16}\mathbb E_x\tau_B\le\mathcal H$, then
\begin{equation}\label{chenapp:eq:mixing-precise}
 t_{\mathrm{mix}}(Q)\le\lceil\mathcal H\rceil+
 \left\lceil\mathcal H\log\frac{21}{5}\right\rceil+
 \left\lceil\frac{17\log2}{2\lambda}\right\rceil.
\end{equation}
\end{chenapplemma}

\begin{proof}
We use the hitting-to-mixing argument of Basu, Hermon, and
Peres~\cite{BHP}, keeping track of the constants.
First, reversibility implies the maximal inequality
\begin{equation}\label{chenapp:eq:maximal}
 \left\|\sup_{j\ge0}|Q^jf|\right\|_{2,\nu}^2
 \le8\|f\|_{2,\nu}^2.
\end{equation}
This is the $L^2$ case of Starr's inequality~\cite{Starr}.
For a short proof, take a stationary chain $(Y_j)$ and a fixed
integer $M$. Reversibility and the Markov property give the reverse martingale
\[
 M_j=\mathbb E[f(Y_0)\mid Y_j,\ldots,Y_M]=Q^jf(Y_j),
 \qquad \mathbb E[M_j\mid Y_0]=Q^{2j}f(Y_0).
\]
Conditional Jensen and Doob's inequality give
\begin{equation}\label{eq:maximal-inequality}
 \left\|\max_{j\le M}|Q^{2j}f|\right\|_{2,\nu}^2
 \le\mathbb E\max_{j\le M}|M_j|^2\le4\|f\|_{2,\nu}^2.
\end{equation}
Apply \eqref{eq:maximal-inequality} also to $Qf$ for odd times, add, and let $M\to\infty$.

Set $k=\lceil17\log2/(2\lambda)\rceil$.
For every event $E$, let $u=\mathbf1_E-\nu(E)$ and
$B_E=\{y:\sup_{j\ge k}|Q^ju(y)|\le1/64\}$.
Laziness, the spectral gap, and \eqref{chenapp:eq:maximal} imply
\[
\nu(B_E^c)\le2(1/64)^{-2}e^{-2\lambda k}\le1/16,
\]
so $\nu(B_E)\ge15/16.$ To bound its hitting-time tail, fix any target $B$ of mass at least
$15/16$ and put $m_j=\mathbb E_x[(\tau_B-j)_+]$.
The Markov property gives
$m_j\le\mathcal H\mathbb P_x(\tau_B>j)$, and therefore
\[
 m_{j+1}=m_j-\mathbb P_x(\tau_B>j)
 \le(1-\mathcal H^{-1})m_j,
 \qquad m_j\le\mathcal H e^{-j/\mathcal H}.
\]
If $\mathcal H=1$, the recursion makes $m_j=0$ for $j\ge1$,
and the same bound holds. A target equal to the state space has zero tail.
For $T=\lceil\mathcal H\rceil+
\lceil\mathcal H\log(21/5)\rceil$,
\[
 \lceil\mathcal H\rceil\mathbb P_x(\tau_B>T)
 \le\sum_{j=T-\lceil\mathcal H\rceil}^{T-1}\mathbb P_x(\tau_B>j)
 \le m_{T-\lceil\mathcal H\rceil}\le\mathcal H\frac5{21}.
\]
Thus $\mathbb P_x(\tau_{B_E}>T)\le5/21$.
On $\{\tau_{B_E}\le T\}$ at least $k$ steps remain until $T+k$.
The strong Markov property now gives
\[
 |Q^{T+k}(x,E)-\nu(E)|
 \le1/64+(1-1/64)\mathbb P_x(\tau_{B_E}>T)
 \le\frac1{64}+\frac{63}{64}\frac5{21}=\frac14.
\]
Taking the supremum over $x,E$ proves \eqref{chenapp:eq:mixing-precise}.

\end{proof}

\begin{proof}[Proof of Theorem~\ref{chenapp:thm:mixing}]
Let \(\mathcal H\) be the right-hand side of \eqref{chenapp:eq:hitting}. Apply Lemma~\ref{chenapp:lem:direct-tv}, using Proposition~\ref{chenapp:prop:hitting} with $\lambda\ge\gamma/n$.
%Apply Proposition~\ref{chenapp:prop:hitting} with $\mathcal H$ be the right-hand side of \ref{chenapp:eq:hitting} and
%Lemma~\ref{chenapp:lem:direct-tv}, using $\lambda\ge\gamma/n$. 
Recall $\ell=\log(e/a)$ and set $L=\log(en)$. Maximizing
$(1+u)/(1+u^2)$ for $0\le u\le1$, and using
$-\log(1-u^2)\ge u^2$, gives
\[
 \frac a\gamma=1+\sqrt{1-a}
 \le\frac{1+\sqrt2}{2}\,\ell.
\]
For $n\ge128$ we have $n^{-1}+\frac12\log2<e^{-1}$, so
$\Xi\le\ell(L-2)$. As $\ell\ge1$ and $D_n/n\le(L+1)/2$,
\eqref{chenapp:eq:hitting} gives
\[
 \frac{\mathcal H}{(n/a)\ell L}
 \le
 \frac e{e-1}\left(\frac32+\log2\right)
 +\frac1L\left[
   \frac{16e(1+\sqrt2)}{15}
   +\frac{e}{2(e-1)}
 \right]
 <4.81.
\]
The right side decreases in $L\ge1+7\log2$, then the stated bound follows
from $e/(e-1)<1.583$, $e<2.719$, $(1+\sqrt2)/2<1.208$,
$\log2<0.694$, and $L>5.85$.
In the same range,
\[
 \frac{17n\log2/(2\gamma)}{(n/a)\ell L}<1.219,
 \qquad \frac3{(n/a)\ell L}<0.005,
 \qquad 1+\log(21/5)<2.436.
\]
Consequently \eqref{chenapp:eq:mixing-precise}, including the three
ceilings, is less than
\[
 (2.436\cdot4.81+1.219+0.005)(n/a)\ell L<13(n/a)\ell L.
\]
For $n<128$, the usual $L^2$ spectral bound instead gives $t_{\mathrm{mix}}(P)
 \le\left\lceil\frac{n}{2\gamma}\log(4|A|)\right\rceil,$ and
\[
 \frac{t_{\mathrm{mix}}(P)}{(n/a)\ell L}
 \le\frac{1+\sqrt2}{4}
 \left[\frac{127\log2}{1+\log127}+\log4\right]+1
 <11.
\]
Here we used $|A|\le2^n$ and the fact that $n/\log(en)$ is
increasing. 
\end{proof}

\section*{Acknowledgments}

Y. Chen thanks Dr.~Fan Chang for sharing this problem and for
providing updates on related work. Y. Chen is partly supported by the National
Natural Science Foundation of China (Grant Nos.~12595294 and 12231002)
and the New Cornerstone Science Foundation (Grant No.~NCI202501). The research of Y. Peres was supported by National Natural Science Foundation of China grant
RFIS-W2531011.

\end{document}